\documentclass[12pt]{article}

\usepackage{amsfonts,amsmath, amssymb}
\usepackage{dutchcal}
\usepackage{hyperref}

\usepackage{xcolor}

\usepackage{enumerate}
\usepackage{epsf,epsfig,amsfonts,a4wide}
\usepackage{amsfonts, mathdots}
\usepackage{amsmath,amssymb,amsthm}
\usepackage{url}
\usepackage{amsmath,amssymb,amsthm}

\usepackage{amssymb}
\usepackage{amsthm}
\usepackage{epsf,epsfig,amsfonts,graphicx}

\newtheorem{Theorem}{Theorem}

\newtheorem{Lemma}{Lemma}[section]

\newtheorem{Ex}{Example}[section]
\newtheorem{Remark}{Remark}[section]

\theoremstyle{remark}

\newcommand{\be}{\begin{equation}}
\newcommand{\ee}{\end{equation}}

\newcommand{\Id}{\textrm{\rm Id}}

\newcommand{\ddd}{\mathrm{d}}

\newcommand{\pd}[2]{\frac{\partial#1}{\partial#2}}

\newcommand{\weg}[1]{}

\title{On operator fields in the upper triangular Toeplitz form}
\author{M. M. Chernin and A. Yu. Konyaev \footnote{Faculty of Mechanics and Mathematics, Moscow State University, and Moscow Center for Fundamental and Applied Mathematics,119992, Moscow Russia}}
\date{August 2025}

\begin{document}

\maketitle

\section{Introduction}

We say that an operator field $L$ is in the upper triangular Toeplitz form \cite{bottcher} if, in given coordinates $u^1, \dots, u^n$, one has
\begin{equation}\label{toeplitz}
L = \left( \begin{array}{cccc}
     g_n & g_{n - 1} & \dots & g_1  \\
     0 &  g_n & \ddots & \vdots \\
     & \ddots &\ddots  & g_{n - 1}\\
     0 & 0 & \dots  & g_n
\end{array}\right).    
\end{equation}
The operator fields in the upper triangular Toeplitz form appear in many branches of mathematics and mathematical physics: in the theory of Frobenius and F-manifolds \cite{antonov, perletti, perletti2}, in reductions of hydrodynamic chains and linearly degenerate dispersionless PDEs \cite{pavlov}, in the study of the soliton gas \cite{vergallo, vergallo2}, and in Nijenhuis geometry \cite{curv, nij4, app_nij5, nij1}, to name a few. 

The main purpose of this article is to investigate the following fundamental problem: 
\begin{enumerate}
    \item[Q:] Given the operator field $L$ in the upper triangular Toeplitz form in coordinates $u^1, \dots, u^n$ with the regularity condition $g_{n - 1} \neq 0$, describe all the coordinate transformations $v(u)$ that preserve this form.
\end{enumerate}
The regularity condition ensures that the operator $L$ is similar to the Jordan block of maximal size in every point in the neighborhood of the coordinate origin. 

It turns out that this problem can be solved in full generality. Informally speaking, all such coordinate transformations are parameterized by a single function of one variable and $n - 1$ functions of two variables. Moreover, given the collection of such functions, one can construct the corresponding transformation $v(u)$ explicitly. The only thing one needs is a recursion procedure, which involves solving triangular linear systems and integration. 

Let us discuss the results in greater detail. To do that, one first needs to introduce certain elements of Nijenhuis geometry. Recall that the Nijenhuis torsion of $L$ is a tensor of type $(1, 2)$, given by the formula
\begin{equation}\label{nij}
\mathcal N_L(\xi, \eta) = L^2[\xi, \eta] + [L\xi, L\eta] - L[L\xi, \eta] - L[\xi, L\eta].    
\end{equation}
Here, $\xi, \eta$ are arbitrary vector fields, and the brackets $[\, , \,]$ stand for the standard vector field commutator. The Haantjes torsion is given by the formula
\begin{equation}\label{haa}
\mathcal H_L(\xi, \eta) = L^2 \mathcal N_L(\xi, \eta) + \mathcal N_L(L\xi, L\eta) - L\mathcal N_L(L\xi, \eta) - L\mathcal N_L(\xi, L\eta).    
\end{equation}
The notation $L \mathcal N_L(\xi, \eta)$ means that we apply $L$ to the vector field $\mathcal N_L(\xi, \eta)$. The operator field $L$ is called the Nijenhuis (Haantjes) operator if its Nijenhuis (Haantjes) torsion identically vanishes. Obviously, if $L$ is Nijenhuis, then it is Haantjes. 

Assume now that $L$ is in Toeplitz form in the given coordinates. Define in the same coordinates 
\begin{equation}\label{jordan}
J = \left( \begin{array}{ccccc}
     0 & 1 & 0 & \dots & 0 \\
     0 & 0 & 1 & \dots & 0 \\
      & & & \ddots & \\
      0 & 0 & 0 & \ddots & 1\\
      0 & 0 & 0 & \dots & 0
\end{array}\right).    
\end{equation}
Obviously, $J$ is Nijenhuis. At the same time $L = g_1 J^{n - 1} + \dots + g_n \Id$, so $L$ is a Haantjes operator (Corollary 3.3 in \cite{bg}). The results of the paper are as follows.

In Theorem \ref{t1} we provide the sufficient conditions for an operator field in Toeplitz form to be Nijenhuis. It also states that in the case of $g_{n - 1} \neq 0$ in formula \eqref{toeplitz} operator, the same conditions are necessary as well. We give an example \ref{ex1} to show that if $g_{n - 1} = 0$, then there exist Nijenhuis operators outside the scope of Theorem \ref{t1}.

The Theorem \ref{t1} reduces the question to the overdetermined system of PDEs with constant coefficients. The initial condition for the corresponding system is at most one function of one variable and $n - 1$ functions of two variables. The Theorem \ref{t2} explicitly solves the aforementioned system: given the initial condition, it produces the corresponding Nijenhuis operator in explicit form. The description of such Nijenhuis operators is an interesting result in its own right.

Theorem \ref{t3} describes the linear system of PDEs that the coordinate change preserving the upper triangular Toeplitz form must satisfy. The idea here is as follows: in given coordinates, fix an arbitrary Nijenhuis operator $M$ in the upper triangular Toeplitz form with $g_{n - 1} \neq 0$ and zeros on the diagonal. Due to the general theory in \cite{nij1}, there exists a coordinate change $v(u)$ such that, in coordinates $v^1, \dots, v^n$, the operator has the form $J$. This yields certain relations on $v(u)$ with respect to the Nijenhuis operator, thus producing the differential equations on $v(u)$.

Finally, we provide an algorithm for resolving the aforementioned system of equations. The theorem \ref{t4} states that the algorithm is correct. It requires only integration and answers the main question posed in the beginning of the paper. We supply the example \eqref{ex2}, which yields all the coordinate transformations that preserve the upper triangular Toeplitz form in dimension four.

The paper is organized as follows: in section \ref{main1}, we formulate the main results, provide remarks, and include necessary examples. The proofs are given in sections \ref{proof1}, \ref{proof2}, \ref{proof3}, and \ref{proof4}. The work is supported by the grant RScF 24-21-00450.


\section{Main results}\label{main1}

We start this section with the following theorem.

\begin{Theorem}\label{t1}
Fix the coordinates $u^1, \dots, u^n$ and assume that the operator field $L$ is in the upper triangular Toeplitz form. The following (equivalent) conditions on $g_i, i = 1, \dots, n$ are sufficient for $L$ to be Nijenhuis:
\begin{enumerate}
    \item Differentials $\ddd g_i$ satisfy the linear system of PDEs
    \begin{equation}\label{eq1}
    \begin{aligned}
    0 & = J^{2*} \ddd g_n, \\
    0 & = J^{2*} \ddd g_{n - 1} - 2 J^* \ddd g_n, \\
    0 & = J^{2*} \ddd g_i - 2 J^* \ddd g_{i + 1} + \ddd g_{i + 2}, \quad i = 1, \dots, n - 2,
    \end{aligned}    
    \end{equation}
    Here $J$ is given by the formula \eqref{jordan}.
    \item Denote $\pd{g}{u}$ as the Jacobi matrix of the system of functions $g_1, \dots, g_n$. Then
    \begin{equation}\label{eq2}
    \Big[\Big[\pd{g}{u}, J\Big], J\Big] = 0,   
    \end{equation}
    where $[, , \,]$ stands for the standard matrix commutator, that is, $[A, B] = AB - BA$. The matrix $J$ is the same as above.
    \item Denote the column $(g_1, \dots, g_n)^T$ by $g$
    \begin{equation}\label{eq3}
        \begin{aligned}
        & J^2 \pd{g}{u^1} = 0, \\
        & J^2 \pd{g}{u^2} - 2 J \pd{g}{u^1} = 0, \\
        & J^2 \pd{g}{u^{i}} - 2 J \pd{g}{u^{i - 1}} + \pd{g}{u^{i - 2}} = 0, \quad i = 3, \dots, n. \\
\end{aligned}
    \end{equation}
\end{enumerate}
If $g_{n - 1} \neq 0$, then the conditions above are also necessary.
\end{Theorem}

The next example shows that without the condition $g_{n - 1} \neq 0$, the equations in Theorem \ref{t1} do not necessarily hold for the Nijenhuis operator in the upper triangular Toeplitz form.

\begin{Ex}\label{ex1}
\rm{
Assume that $n = 3$ and $g_{2} \equiv 0$. Let $L$ be a Nijenhuis operator in the upper triangular Toeplitz form. By direct calculations we get
$$
\begin{aligned}
\mathcal N_L(\partial_1, \partial_2) = 0, \quad  \mathcal N_L(\partial_1, \partial_3) = 2 g_1 \pd{g_3}{u^1} \partial_1, \quad  \mathcal N_L(\partial_2, \partial_3) = g_1 \pd{g_3}{u^2} \partial_1 + g_1 \pd{g_3}{u^1} \partial_2.
\end{aligned}
$$
Vanishing of the Nijenhuis torsion implies that either $g_3$ does not depend on $u^1, u^2$ or $g_1$ identically vanishes. We get two classes of Nijenhuis operators
$$
A = \left(\begin{array}{ccc}
     a(u^3) & 0 & b(u^1, u^2, u^3) \\
     0 & a(u^3) & 0 \\
     0 & 0 & a(u^3)
\end{array}\right), \quad B = \left(\begin{array}{ccc}
     c(u^1, u^2, u^3) & 0 & 0 \\
     0 & c(u^1, u^2, u^3) & 0 \\
     0 & 0 & c(u^1, u^2, u^3)
\end{array}\right).
$$
Here $a$ is a function of single variable and $b, c$ are arbitrary functions of three variables. Both Nijenhuis operators do not satisfy the equations in Theorem \ref{t1}. $\blacksquare$
}    
\end{Ex}

Let us recall some facts about matrix-valued functions (chapter 5 in \cite{gantmaher}). Consider a smooth function $f$ of two variables and a pair of polynomials
$$
p(t) = c_1 t^{n - 1} + \dots + c_n \quad \text{and} \quad q(t) = m_1 t^{n - 1} + \dots + m_n.
$$
Take the following decomposition (we change the numeration for the first $n$ terms) in $t = 0$
\begin{equation}\label{decomp}
f\big(p(t), q(t)\big) = f_{n} + f_{n - 1} t + \dots + f_1 t^{n - 1} + \{\text{terms, containing $t^k$ for $k \geq n$}\} .    
\end{equation}
Each $f_i$ is a function in $c_i, m_j$. Now consider a pair of operators in Toeplitz form
$$
P = \left( \begin{array}{ccccc}
     c_n & c_{n - 1} &\dots  & c_2 & c_1  \\
     0 & c_n & \ddots & & c_2 \\
     &  &\ddots & \ddots & \vdots\\
     0 & 0 & \dots & c_n & c_{n - 1} \\
     0 & 0 & \dots & 0 & c_n
\end{array}\right) \quad \text{and} \quad Q = \left( \begin{array}{ccccc}
     m_n & m_{n - 1} &\dots  & m_2 & m_1  \\
     0 & m_n & \ddots & & m_2 \\
     &  &\ddots & \ddots & \vdots\\
     0 & 0 & \dots & m_n & m_{n - 1} \\
     0 & 0 & \dots & 0 & m_n
\end{array}\right).
$$
The matrix-valued function $f(P, Q)$ is given by the formula
$$
f(P, Q) = \left( \begin{array}{ccccc}
     f_n(c, m) & f_{n - 1}(c, m) & \dots & f_2(c, m) & f_1 (c, m)  \\
     0 & f_n(c, m) & \ddots & & f_2(c, m) \\
     &  &\ddots & \ddots & \vdots\\
     0 & 0 & \dots & f_n(c, m) & f_{n - 1}(c, m) \\
     0 & 0 & \dots & 0 & f_n(c, m)
\end{array}\right),
$$
where $f_i(c, m)$ are the coefficients of decomposition \eqref{decomp}, depending on $c, m$. 

\begin{Theorem}\label{t2}
Fix the coordinates $u^1, \dots, u^n$ and consider $J$ in the form
$$
J = \left( \begin{array}{ccccc}
     0 & 1 & 0 & \dots & 0 \\
     0 & 0 & 1 & \dots & 0 \\
      & & & \ddots & \\
      0 & 0 & 0 & \ddots & 1\\
      0 & 0 & 0 & \dots & 0
\end{array}\right) 
$$
and a pair of operator fields
$$
P = u^1 J^{n - 1} + \dots + u^n\Id \quad \text{and} \quad Q = (n - 1) u^1 J^{n - 2} + \dots + 2u^{n - 2} J + u^{n - 1} \Id.
$$
Fix an arbitrary collection of $n - 1$ functions of two variables $f_1, \dots, f_{n - 1}$ and one function $f_n$ of a single variable. Define  operator field
\begin{equation}\label{solution}
L = f_1(P, Q) J^{n - 1} + \dots + f_{n - 1}(P, Q) J + f_n(P).     
\end{equation}
By definition, $L$ is in the upper triangular Toeplitz form. Then the operator $L$ is a Nijenhuis operator, and any Nijenhuis operator in upper triangular form for $g_{n - 1} \neq 0$ can be obtained in this manner from the appropriate collection of smooth functions $f_i$ with $f_{n - 1} (0, 0) \neq 0$.   
\end{Theorem}

The next example provides an explicit calculation in dimension four.

\begin{Ex}\label{exm}
\rm{
Consider case $n = 4$. Fix three functions of two variables $a, b, c$ and one function of single variable $d$. The polynomials $p(t), q(t)$ in this case are (we use lower indices for coordinates in this example)
$$
p(t) = u_1 t^3 + u_2 t^2 + u_3 t + u_4, \quad q(t) = 3 u_1 t^2 + 2 u_2 t + u_3.
$$
The decomposition of $d(p)$ takes form
$$
\begin{aligned}
d(p(t)) = d(u_4) + u_3 d'(u_4) t + & \Big( d'(u_4) u_2 + \frac{1}{2}d''(u_4) u_3^2\big) t^2 + \\
& + \Big( d'(u_4) u_1 + d''(u_4) u_2 u_3 + \frac{1}{6}d'''(u_4) u_3^3\Big)t^3 + \dots.
\end{aligned}
$$
The decomposition of $a(p, q)$ takes form
$$
\begin{aligned}
a(p(t), q(t)) & = a(u_4, u_3) + \Big(a'_p (u_4, u_3) u_3 + 2 a'_q (u_4, u_3) u_2 \Big)t + \\
& + \Big( a'_p (u_4, u_3)u_2 + 3 a'_q (u_4, u_3)u_1 + \frac{1}{2} a''_{pp} u_3^2 + 2 a''_{pq}u_3 u_2 + 2 a''_{qq} u_2^2\Big)t^2 + \dots.
\end{aligned}
$$
For $b, c$ we have a similar decomposition. For operator fields $P, Q$ we get
$$
d(P) = \left( \begin{array}{cccc}
     d(u_4) & u_3 d'(u_4) & d'(u_4) u_2 + \frac{1}{2}d''(u_4) u_3^2 & d'(u_4) u_1 + d''(u_4) u_4 u_3 + \frac{1}{6}d'''(u_4) u_3^3\\
     0 & d(u_4) & u_3 d'(u_4) & d'(u_4) u_2 + \frac{1}{2}d''(u_4) u_3^2\\
     0 & 0 & d(u_4) & u_3 d'(u_4)\\
     0 & 0 & 0 & d(u_4) \\
\end{array}\right),
$$
$$
c(P, Q) J = \left( \begin{array}{cccc}
     0 & c(u_4, u_3) & c'_p (u_4, u_3) u_3 + 2 c'_q (u_4, u_3) u_2 & \scalebox{0.7}{$
     \begin{aligned}
       c'_p (u_4, u_3)u_2 & + 3 c'_q (u_4, u_3)u_1 + \\
        & + \frac{1}{2} c''_{pp} u_3^2 + 2 c''_{pq}u_3 u_2 + 2 c''_{qq} u_2^2  
     \end{aligned}
          $}\\
     0 & 0 & c(u_4, u_3) & \begin{array}{c}
          c'_p (u_4, u_3) u_3 + 2 c'_q (u_4, u_3) u_2 \\
     \end{array}\\
     0 & 0 & 0 & c(u_4, u_3) \\
     0 & 0 & 0 & 0 \\
\end{array}\right),
$$
$$
b(P, Q) J^2 = \left(\begin{array}{cccc}
     0 & 0 & b(u_4, u_3) & \scalebox{0.7}{$b'_p (u_4, u_3) u_3 + 2 b'_q (u_4, u_3) u_2$}\\
     0 & 0 & 0 & b(u_4, u_3) \\
     0 & 0 & 0 & 0 \\
     0 & 0 & 0 & 0 \\
\end{array}\right), a(P, Q)J^3 = \left(\begin{array}{cccc}
     0 & 0 & 0 & a(u_4, u_3)\\
     0 & 0 & 0 & 0 \\
     0 & 0 & 0 & 0 \\
     0 & 0 & 0 & 0 \\
\end{array}\right).
$$
The general form of $L$ in dimension four is given by formula $L = a(P, Q)J^3 + b(P, Q) J^2 + c(P, Q) J + d(P)$. We do not write it in general matrix form due to the sizes of the formulas. $\blacksquare$
}    
\end{Ex}

\begin{Theorem}\label{t3}
Fix the coordinates $u^1, \dots, u^n$ and let $L$ be an operator field in the upper triangular Toeplitz form \eqref{toeplitz} with the condition $g_{n - 1} \neq 0$. Consider $J, P, Q$ as the same as in theorem \ref{t2}. Fix $n - 1$ functions of two variables $f_1, \dots, f_{n - 1}$ with $f_{n - 1} (0, 0) \neq 0$ and take 
$$
M = f_1(P, Q) J^{n - 1} + \dots + f_{n - 1}(P, Q)J. 
$$
Consider the solution $v^1, \dots, v^n$ of the system
\begin{equation}\label{sys}
\begin{aligned}
M^* \ddd v^i = \ddd v^{i + 1}, \quad i = 1, \dots, n - 1
\end{aligned}    
\end{equation}
with $\pd{v^1}{u^1} \neq 0$. Then 
\begin{enumerate}
    \item Functions $v^i, i = 1, \dots, n$ define coordinate transformation;
    \item This transformation preserves the upper triangular Toeplitz form of $L$;
    \item Any transformation that preserves the upper triangular Toeplitz form of $L$ can be obtained in such a way for an appropriate choice of functions $f_i$. 
\end{enumerate}
\end{Theorem}

\begin{Remark}\label{r1}
\rm{
The transformation from $M$ to $J$ from Theorem \ref{t3} is not unique and is defined modulo the transformation, that preserves $J$. They can be constructed as follows: pick $n$ functions $h_1, \dots, h_n$ of single variable with $h_1'(0) \neq 0$. Let $P$ be the same as in theorem \ref{t2}. Consider operator field in the upper triangular Toeplitz form
$$
h_1(P)J^{n - 1} + \dots + h_n(P) \Id = \left( \begin{array}{ccccc}
     w_n & w_{n - 1} &\dots  & w_2 & w_1  \\
     0 & w_n & \ddots & & w_2 \\
     &  &\ddots & \ddots & \vdots\\
     0 & 0 & \dots & w_n & w_{n - 1} \\
     0 & 0 & \dots & 0 & w_n
\end{array}\right). 
$$
Then functions $w(u)$ define a coordinate transformation, that preserves $J$. Note that any transformation, that preserves $J$, can be obtained in such a way. $\blacksquare$
}    
\end{Remark}

The next algorithm allows one to solve the system \eqref{sys}. The initial data of the algorithm is $M$ from Theorem \ref{t3}.
\begin{enumerate}
    \item[Step 0:] Pick an arbitrary function $q(u^n) \neq 0$ and define $v^n = \int_0^{u^n} q(\tau) \ddd \tau$;
    \item[Step 1:] Now assume that we have constructed $v^{n - i}$. If $i = n - 1$, then the algorithm stops, and we have constructed the solution of the system \eqref{sys} with $\pd{v_1}{u^1} \neq 0$. If $i < n - 1$, then continue;
    \item[Step 2:] Construct $\omega$ in the form
    $$
    \omega = \omega_1 \ddd u^1 + \dots + \omega_{n - 1} \ddd u^{n - 1},
    $$
    such that $M^* \omega = \ddd v^{n - i}$. This is a solution to a triangular system of linear equations (in the algebraic sense).
\item[Step 3:] Construct
     \begin{equation}\label{form2}
     \begin{aligned}
     v^{n - i - 1} = \int_0^{u^1} \omega_1(t, &  u^2, \dots, u^n) \ddd t + \int_0^{u^2} \omega_2(0, t, \dots, u^n) \ddd t  + \dots + \\
     & + \int_0^{u^{n - 1}} \omega_{n - 1}(0, 0, \dots, t,  u^n) \ddd t  + r(u^n).    
     \end{aligned} 
    \end{equation}
    Here the choice of $r(u^n)$ is arbitrary.
    \item[Step 4:] Go to step 1.
\end{enumerate}

The algorithm implies that the system \eqref{sys} can be solved explicitly, using only integrations.

\begin{Theorem}\label{t4}
Fix coordinates $u^1, \dots, u^n$ and let $M$ be an operator constructed in Theorem \ref{t3}. Then
\begin{enumerate}
    \item The algorithm is correct. That is, for any initial data, it produces the collection of functions $v^1, \dots, v^n$;
    \item Functions $v^1, \dots, v^n$ satisfy the system \eqref{sys} in Theorem \ref{t3} with the condition $\pd{v_1}{u_1} \neq 0$;
    \item All the solutions of the system \eqref{sys} in Theorem \ref{t3} with the condition $\pd{v_1}{u_1} \neq 0$ can be obtained via the appropriate choice of the function $q$ and the "constants of integration" in the formula \eqref{form2}.
\end{enumerate}
\end{Theorem}

Let us give the example of the application of our algorithm:
\begin{Ex}\label{ex2}
\rm{
Fix coordinates $u_1, u_2, u_3, u_4$ (we use lower indices for coordinates and slightly different notations then in the proof of the Theorem \ref{t4}) and start with initial data for $M$. The preliminary calculations were performed in example \eqref{ex1}, so we just write the result for the parameter $d \equiv 0$:
$$
M = a(P, Q)J^3 + b(P, Q) J^2 + c(P, Q) J = \left( \begin{array}{cccc}
     0 & c(u_4, u_3) & r(u_4, u_3, u_2) & m(u_4, u_3, u_2, u_1)  \\
     0 & 0 & c(u_4, u_3) & r(u_4, u_3, u_2) \\
     0 & 0 & 0 & c(u_4, u_3) \\
     0 & 0 & 0 & 0
\end{array}\right),
$$
where
$$
\begin{aligned}
r(u_4, u_3, u_2) & = c'_p (u_4, u_3) u_3 + 2 c'_q (u_4, u_3) u_2 + b(u_4, u_3), \\
m(u_4, u_3, u_2, u_1) & = c'_p (u_4, u_3)u_2 + 3 c'_q (u_4, u_3)u_1 + \frac{1}{2} c''_{pp} u_3^2 + \\
& + 2 c''_{pq}u_3 u_2 + 2 c''_{qq} u_2^2 + b'_p (u_4, u_3) u_3 + 2 b'_q (u_4, u_3) u_2 + a(u_4, u_3).
\end{aligned}
$$
We start with step 0. We get $v_4 = \int_0^{u_4} q(\tau) \ddd \tau$. We pass the test in step 1 and move to the step 2. In the 1-form $\omega^3$ the only non zero coefficient is $\omega_3$. It has the form
$$
\omega^4_3 = \frac{s_4(u_4)}{c(u_4, u_3)}
$$
for some non-zero function $s_4$ of a single variable. The integral formula for $v_3$ yields
$$
v_3(u_4, u_3) = \int_0^{u_3} \frac{s_4(u_4)}{c(u_4, \tau)} \ddd \tau + s_3 (u_4),
$$
where $s_3$ is an arbitrary function of single variable. We again pass the test in step 2 and go the step 3. In this case $\omega^2$ is found from the system of linear equations $M^* \omega^2 = \ddd v_3$
$$
\begin{aligned}
\begin{aligned}
& c(u_4, u_3) \omega^2_2 = \pd{v_3}{u_3}, \\
& r(u_4, u_3, u_2) \omega^2_2 + c(u_4, u_3) \omega^2_3 = \pd{v_3}{u_4}.
\end{aligned}
\end{aligned}
$$
Resolving this system, we get
$$
\omega^2_2 = \frac{s_4(u_4)}{c^2(u_4, u_3)} \quad \text{and} \quad \omega^2_3 = \frac{1}{c(u_4, u_3)} \Bigg( \int_0^{u_3}\Big( \frac{s_4(u_4)}{c(u_4, \tau)}\Big)'_{u_4} \ddd \tau + s'_3(u_4) - r(u_4, u_3, u_2) \frac{s_4(u_4)}{c^2(u_4, u_3)}\Bigg).
$$
The integral formula for $v_2$ in this case is
$$
\begin{aligned}
v_2(u_4, u_3, u_2) & = \frac{s_4(u_4)}{c^2(u_3, u_3)}u^2 + s_2(u_4) + \\
& + \int^{u_3}_0 \Bigg[ \frac{1}{c(\rho, u_4)} \Bigg( \int_0^{\rho}\Big( \frac{s_4(u_4)}{c(u_4, \tau)}\Big)'_{u_4} \ddd \tau + s'_3(u_4) - r(u_4, \rho, 0) \frac{s_4(u_4)}{c^2(u_4, \rho)}\Bigg)\Bigg] \ddd \rho .   
\end{aligned}
$$
Finally, during our last iteration we arrive to the following equation on components of $\omega^1$:
$$
\begin{aligned}
\begin{aligned}
& c(u_4, u_3) \omega^1_1 = \pd{v_2}{u_2}, \\
& r(u_4, u_3, u_2) \omega^1_1 + c(u_4, u_3) \omega^1_2 = \pd{v_2}{u_3}, \\
& m(u_4, u_3, u_2, u_1) \omega^1_1 + r(u_4, u_3, u_2) \omega^1_2 + c(u_4, u_3) \omega^1_3 = \pd{v_2}{u_4}.
\end{aligned}
\end{aligned}
$$
Resolving this system, we get (we do not substitute the formulas for $\pd{v_2}{u_3}$ and $\pd{v_2}{u_2}$ due to the enormous complexity)
$$
\begin{aligned}
\omega^1_1 & = \frac{s_4(u_4)}{c^3(u_4, u_3)}, \\
\omega^1_2 & = \frac{1}{c(u_4, u_3)} \pd{v_2 (u_4, u_3, u_2)}{u_3} - r(u_4, u_3, u_2)\frac{s_4(u_4)}{c^4(u_4, u_3)}, \\
\omega^1_3 & = \frac{1}{c(u_4, u_3)} \pd{v_2 (u_4, u_3, u_2)}{u_4} - \frac{r(u_4, u_3, u_2)}{c^2(u_4, u_3)} \pd{v_2 (u_4, u_3, u_2)}{u_3} + \\
& + r^2(u_4, u_3, u_2)\frac{s_4(u_4)}{c^5(u_4, u_3)} - m(u_4, u_3, u_2, u_1) \frac{s_4(u_4)}{c^4(u_4, u_3)}.
\end{aligned}
$$
The integral for function $v_1$ takes form
$$
\begin{aligned}
v_1(u_4, u_3, u_2, u_1) & = \frac{s_4(u_4)}{c^3(u_4, u_3)} u_1 + \frac{1}{c(u_4, u_3)} \int^{u_2}_0 \pd{v_2(u_4, u_3, \tau)}{u_3} \ddd \tau - \frac{s_4(u_4)}{c^4(u_4, u_3)} \int_0^{u_2} r(u_4, u_3, \tau) \ddd \tau + \\
& + \int_0^{u_3} \frac{1}{c(u_4, \tau)} \pd{v_2 (u_4, \tau, 0)}{u_4} \ddd \tau - \int_0^{u_3} \frac{r(u_4, \tau, 0)}{c^2(u_4, \tau)} \pd{v_2 (u_4, \tau, 0)}{u_3} \ddd \tau + \\
& + \int_0^{u_3} r^2(u_4, \tau, 0)\frac{s_4(u_4)}{c^5(u_4, \tau)} \ddd \tau - \int_0^{u_3} m(u_4, \tau, 0, 0) \frac{s_4(u_4)}{c^4(u_4, \tau)} \ddd \tau \;+ s_1(u^4) .
\end{aligned}
$$
One can see that the complexity of formulas rises fast. $\blacksquare$
}    
\end{Ex}


\section{Proof of Theorem \ref{t1}}\label{proof1}

We start the proof by recalling the notion of $gl-$regularity. We say that a linear operator $L: V^n \to V^n$ is $gl-$regular if one of the following equivalent conditions holds:
\begin{enumerate}
    \item For any $M$, such that $LM - ML = 0$, there exist constants $g_1, \dots, g_n$, such that $M = g_1 L^{n - 1} + \dots + g_n \Id$;
    \item There exists a vector $\xi$ such that $\xi, L\xi, \dots, L^{n - 1}\xi$ are linearly independent. The corresponding vector is called cyclic;
    \item For every eigenvalue of $L$, there is exactly one Jordan block that corresponds to it.
\end{enumerate}
For $L$ in the upper triangular Toeplitz form, the $gl$-regularity is equivalent to the fact that $L$ is similar to the Jordan block of maximal size. Equivalently, this is the condition $g_{n - 1} \neq 0$.

Let $L, M$ be a pair of operator fields. Define 
\begin{equation}\label{ii:eq2}
\langle L, M \rangle (\xi, \eta) = LM[\xi, \eta] + [L\xi, M\eta] - L[\xi, M\eta] - M[L\xi, \eta].
\end{equation}
The r.h.s. of this formula defines a tensor field of type $(1, 2)$ if and only if $LM - ML = 0$ (formula 3.9 in \cite{nijenhuis}). We start with a technical lemma related to this operation. 

\begin{Lemma}\label{lm0}
Let $L, M$ be a pair of operator fields such that $LM - ML = 0$ and $f$ are arbitrary functions. Then the following tensor identities hold
\begin{equation}\label{ii:tensor}
    \begin{aligned}
       & \langle f L, M \rangle = f \langle L, M \rangle + M L \otimes \ddd f - L \otimes M^* \ddd f, \\
       & \langle L, f M \rangle = f \langle L, M \rangle + L^* \ddd f \otimes M - \ddd f \otimes ML. \\
\end{aligned}
\end{equation}    
\end{Lemma}
\begin{proof}
For arbitrary vector fields $\xi, \eta$, we get
$$
\begin{aligned}
 \langle f L, M \rangle (\xi, \eta) & = f LM[\xi, \eta] + [f L\xi, M\eta] - f L[\xi, M\eta] - M[f L\xi, \eta] = \\
 & = f \langle L, M \rangle(\xi, \eta) + \mathcal L_\eta f \, ML \xi - \mathcal L_{M\eta} L\xi = \\
 & = (f\langle L, M \rangle + ML \otimes \ddd f - L \otimes M^*\ddd f) (\xi, \eta).
\end{aligned}
$$
Here $[\xi, \eta]$ stands for the standard commutator, $\mathcal L_\xi$ is a Lie derivative, and $\mathcal L_\xi \eta = [\xi, \eta]$. This proves the first property. The second is proved in a similar manner.
\end{proof}

Now fix coordinates and consider $L$ in the form \eqref{toeplitz}. Define tensor fields $T_L$ and $M_L$ of type $(1, 2)$:
$$
\begin{aligned}
T_L & = \ddd g_1 \otimes J^{n - 1} + \dots + \ddd g_n \otimes \Id + J^{n - 1} \otimes \ddd g_1 + \dots + \Id \otimes \ddd g_n, \\
M_L & = J^*\ddd g_1 \otimes J^{n - 1} + \dots + J^*\ddd g_n \otimes \Id + J^{n - 1} \otimes \ddd g_2 + \dots + J \otimes \ddd g_n - \\
& - \ddd g_2 \otimes J^{n - 1} - \dots - \ddd g_n \otimes J - J^{n - 1} \otimes J^* \ddd g_1 - \dots - \Id \otimes J^*\ddd g_n.
\end{aligned}
$$
Here $J$ is \eqref{jordan}. The tensor $T_L$ is symmetric in lower indices, while $M_L$ is skew symmetric in lower indices, and $$ M_L (\xi, \eta) = T_L(J\xi, \eta) - T_L(\xi, J \eta)
$$
is defined for arbitrary vector fields $\xi, \eta$. The following lemma holds.

\begin{Lemma}\label{lm1}
If $M_L = 0$, then $L$ is Nijenhuis. If $g_{n - 1} \neq 0$ holds locally, then the condition is also necessary.
\end{Lemma}    
\begin{proof}
Notice that the vanishing of $M_L$ is equivalent to the fact that $J$ is self-adjoint with respect to $T_L$. So, let us prove this. Define
$$
A = \ddd g_1 \otimes J^{n - 1} + \dots + \ddd g_n \otimes \Id, \quad B = J^{n - 1} \otimes \ddd g_1  + \dots + \Id \otimes \ddd g_n.
$$
We have
$$
A(L\xi, \eta) + B(L \xi, \eta) = T_L(L\xi, \eta) \quad \text{and} \quad A(\xi, L\eta) + B(\xi, L\eta) = T_L(\xi, L\eta).
$$
Now consider
$$
\begin{aligned}
\langle L, L \rangle & = \langle \sum_{i = 1}^n g_i J^{n - i}, L \rangle = \sum_{i = 1}^n \Bigg( g_i \langle J^{n - i}, L \rangle + LJ^{n - i} \otimes \ddd g_i - J^{n - i} \otimes L^* \ddd g_i\Bigg) = \\
& = \sum_{i = 1}^n \sum_{j = 1}^n \Bigg(g_i J^{* n - i} \ddd g_j \otimes J^{n - j} -  g_i \ddd g_j \otimes J^{n - i} J^{n - j}\Bigg) + B(L \cdot, \cdot) - B(\cdot, L\cdot)  = \\
& = A(L\cdot, \cdot) - A(\cdot, L \cdot) + B(L \cdot, \cdot) - B(\cdot, L\cdot) = T_L(L\cdot, \cdot) - T_L(\cdot, L\cdot).
\end{aligned}
$$
Here we used Lemma \eqref{lm0}. Thus, $\langle L, L \rangle = 0$ if and only if
$$
T_L(L\xi, \eta) = T_L(\xi, L\eta)
$$
for all vector fields $\xi, \eta$. At the same time (as it was mentioned in the beginning of the proof) $M_L = 0$ implies that $T_L(J\xi, \eta) = T_L(\xi, J\eta)$, which in turn implies the above condition. Thus, the vanishing of $M_L$ is a sufficient condition for $L$ in the upper triangular Toeplitz form to be Nijenhuis.

If $g_{n - 1} \neq 0$, then $L$ is similar to the Jordan block at each point. As $LJ - JL = 0$, we find that there exists a polynomial $p(t)$ with functional coefficients such that $p(L) = J$. If $L$ is self-adjoint with respect to $T_L$, then $J$ is self-adjoint with respect to $T_L$. 

Thus, $L$ is Nijenhuis if and only if $J$ is self-adjoint with respect to $T_L$. Due to the comment in the beginning of the proof, this implies the statement of the lemma \ref{lm1}.
\end{proof}

\begin{Lemma}\label{lm2}
The condition $M_L = 0$ for $L$ in the upper triangular Toeplitz form is equivalent to the conditions \eqref{eq1}.
\end{Lemma}
\begin{proof}
Throughout this proof, we will use $\langle \, , \,\rangle$ as a pairing of vector and covector fields, not the tensor operation \eqref{ii:tensor}. First, let us show that the conditions \eqref{eq1} are sufficient. For arbitrary vector fields $\xi, \eta$, we have
$$
\begin{aligned}
& M_L(J\xi, \eta) - M_L(\xi, J\eta) = \sum_{i = 1}^n \Big( \langle J^* \ddd g_i, J \xi \rangle J^{n - i} \eta  - \langle J^* \ddd g_i, \eta \rangle J^{n - i + 1} \xi\Big) - \\
& - \sum_{i = 1}^{n - 1} \Big( \langle \ddd g_{i + 1}, J \xi \rangle J^{n - i} \eta  - \langle \ddd g_{i + 1}, \eta \rangle J^{n - i + 1} \xi\Big) - \sum_{i = 1}^n \Big( \langle J^* \ddd g_i, \xi \rangle J^{n - i + 1} \eta  - \langle J^* \ddd g_i, J \eta \rangle J^{n - i} \xi\Big) + \\
& + \sum_{i = 1}^{n - 1} \Big( \langle \ddd g_{i + 1}, \xi \rangle J^{n - i + 1} \eta  - \langle \ddd g_{i + 1}, J \eta \rangle J^{n - i} \xi\Big) = \sum_{i = 1}^n \langle J^* \ddd g_i, J \xi \rangle J^{n - i} \eta - \sum_{i = 1}^{n - 1} \langle J^* \ddd g_{i + 1}, \eta \rangle J^{n - i} \xi - \\
& - \sum_{i = 1}^{n - 1} \langle \ddd g_{i + 1}, J \xi \rangle J^{n - i} \eta + \sum_{i = 1}^{n - 2} \langle \ddd g_{i + 2}, \eta \rangle J^{n - i} \xi - \sum_{i = 1}^{n - 1} \langle J^* \ddd g_{i + 1}, \xi \rangle J^{n - i} \eta  + \sum_{i = 1}^{n}\langle J^* \ddd g_i, J \eta \rangle J^{n - i} \xi + \\
& + \sum_{i = 1}^{n - 2} \langle \ddd g_{i + 2}, \xi \rangle J^{n - i} \eta  -\sum_{i = 1}^{n - 1} \langle \ddd g_{i + 1}, J \eta \rangle J^{n - i} \xi = \langle J^{2*} \ddd g_n, \xi \rangle \eta + \langle J^{2*} \ddd g_{n - 1} - 2 J^* \ddd g_n , \xi \rangle J \eta + \\
& + \sum_{i = 1}^{n - 2}\langle J^{2*} \ddd g_i - 2 J^* \ddd g_{i + 1} + \ddd g_{i + 2}, \xi \rangle J^{n - i}\eta + \langle J^{2*} \ddd g_n, \eta \rangle \xi + \langle J^{2*} \ddd g_{n - 1} - 2 J^* \ddd g_n , \eta \rangle J \xi + \\
& + \sum_{i = 1}^{n - 2}\langle J^{2*} \ddd g_i - 2 J^* \ddd g_{i + 1} + \ddd g_{i + 2}, \eta \rangle J^{n - i}\xi = 0.
\end{aligned} 
$$
Thus, we see that due to \eqref{eq1}, $J$ is self-adjoint with respect to $M_L$. For a pair of basis vector fields $\partial_k, \partial_m$ with $k \leq m$ we have two cases:
\begin{enumerate}
    \item $k + m = 2s$ for some $s$. In this case
    $$
    M_L(\partial_k, \partial_m) = M_L(J^k \partial_n, J^m \partial_n) = M_L(J^{k + 1} \partial_n, J^{m - 1} \partial_n) = \dots = M_L(J^s\partial_n, J^s \partial_n) = 0.
    $$
    The last equality holds due to the fact that $M_L$ is skew symmetric in lower indices.
    \item $k + m = 2s + 1$ for some $s$. Define $\bar M_L(\xi, \eta) = M_L(J\xi, \eta)$. In this case
    $$
    \bar M_L(\xi, \eta) = M_L(J\xi, \eta) = - M_L(\eta, J\xi) = - M_L(J\eta, \xi) = - \bar M_L (\eta, \xi).
    $$
    That is, $\bar M_L$ is skew symmetric. At the same time $M_L(\partial_k, \partial_m)= \bar M_L(\partial_{k+1}, \partial_m)$. Applying the previous arguments, we get $M_L(\partial_k, \partial_m) = 0$ as well.
\end{enumerate}
Thus, we have shown that $M_L = 0$ and the conditions \eqref{eq1} are sufficient. Now let us show that they are necessary. 

For operator $J$ the vector field $\xi$ is cyclic if and only if its $n-$th coordinate is non-zero. In particular, the set of cyclic vectors is everywhere dense in every tangent space. Take cyclic $\xi$ and consider
$$
\begin{aligned}
M_L(\xi, J\xi) = & \langle J^* \ddd g_n, \xi \rangle J\xi - \langle J^{2*} \ddd g_n, \xi \rangle \xi + \sum_{i = 1}^{n - 2} \langle J^* \ddd g_{i + 1} - \ddd g_{i + 2}, \xi \rangle J^{n - i} \xi - \\
- & \sum_{i = 1}^{n - 1} \langle J^{2*} \ddd g_i - J^*\ddd g_{i + 1}, \xi \rangle J^{n - i}\xi = - \langle J^{2*} \ddd g_n, \xi \rangle \xi - \langle J^{2*} \ddd g_{n - 1} - 2 J^* \ddd g_n , \xi \rangle J\xi - \\
- & \sum_{i = 1}^{n - 2} \langle J^{2*} \ddd g_i - 2 J^* \ddd g_{i + 1} + \ddd g_{i + 2}, \xi \rangle J^{n - i}\xi = 0.
\end{aligned}
$$
As $\xi$ is cyclic, we get that all the coefficients are zero; that is
$$
\begin{aligned}
& \langle J^{2*} \ddd g_n, \xi \rangle = 0, \\
& \langle J^{2*} \ddd g_{n - 1} - 2 J^* \ddd g_n , \xi \rangle = 0, \\
& \langle J^{2*} \ddd g_i - 2 J^* \ddd g_{i + 1} + \ddd g_{i + 2}, \xi \rangle = 0, \quad i = 1, \dots, n - 2.
\end{aligned}
$$
As cyclic vectors are dense, by continuity, the $1-$forms in the pairings are identically zero. We obtain exactly the formulas \eqref{eq1}. Thus, the Lemma is proved.
\end{proof}

To finish the proof of Theorem \ref{t1} we need to show that conditions \eqref{eq2} and \eqref{eq3} are equivalent to \eqref{eq1}. We start with condition \eqref{eq2}. Opening the brackets in the l.h.s. of \eqref{eq2}, we get
$$
\Big\{\Big\{\pd{g}{u}, J\Big\}, J\Big\} = J^2 \pd{g}{u} - 2 J \pd{g}{u} J + \pd{g}{u} J^2.
$$
Thus, \eqref{eq2} is equivalent to the following
$$
\pd{g}{u}^* J^{2*} - 2 J^* \pd{g}{u}^* J^* + J^{2*} \pd{g}{u}^* = 0.
$$
Applying the l.h.s. to basis differential forms $\ddd u^i$ and keeping in mind that $\pd{g}{u}^* \ddd u^i = \ddd g_i$, we obtain
$$
\begin{aligned}
& \Big(\pd{g}{u}^* J^{2*} - 2 J^* \pd{g}{u}^* J^* + J^{2*} \pd{g}{u}^*\Big) \ddd u^n = J^{2*} \ddd g_n, \\
& \Big(\pd{g}{u}^* J^{2*} - 2 J^* \pd{g}{u}^* J^* + J^{2*} \pd{g}{u}^*\Big) \ddd u^{n - 1} = J^{2*} \ddd g_{n - 1} - 2 J^* \ddd g_n, \\
& \Big(\pd{g}{u}^* J^{2*} - 2 J^* \pd{g}{u}^* J^* + J^{2*} \pd{g}{u}^*\Big) \ddd u^i = J^{2*} \ddd g_i - 2 J^* \ddd g_{i + 1} + \ddd g_{i + 2}, \quad i = 1, \dots, n - 2.
\end{aligned}
$$
The r.h.s. of these equations is exactly \eqref{eq1}. At the same time \eqref{eq2} is equivalent to the vanishing of all l.h.s. Thus, both conditions are equivalent.

Now let us proceed to the \eqref{eq3}. By construction $\pd{g}{u} \partial_i = \pd{g}{u^i}$. At the same time we have
$$
\begin{aligned}
& \Big(J^2 \pd{g}{u} - 2 J \pd{g}{u} J + \pd{g}{u} J^2 \Big) \partial_1 = J^2 \pd{g}{u^1}, \\
& \Big(J^2 \pd{g}{u} - 2 J \pd{g}{u} J + \pd{g}{u} J^2 \Big) \partial_2 = J^2 \pd{g}{u^2} - 2 J \pd{g}{u^1}, \\
& \Big(J^2 \pd{g}{u} - 2 J \pd{g}{u} J + \pd{g}{u} J^2 \Big) \partial_i = J^2 \pd{g}{u^{i}} - 2 J \pd{g}{u^{i - 1}} + \pd{g}{u^{i - 2}}, \quad i = 3, \dots, n. \\
\end{aligned}
$$
Applying similar arguments, we obtain \eqref{eq3}. Theorem \ref{t1} is proved.


\section{Proof of Theorem \ref{t2}}\label{proof2}

We assume that the coordinates $u^1, \dots, u^n$ are fixed and the operator field $J$ is of the form \eqref{jordan}. Let us start with a technical lemma.

\begin{Lemma}\label{lm4}
Conditions \eqref{eq3} are equivalent to the following system
\begin{equation}\label{mod2}
\begin{aligned}
\pd{g}{u^{n - k}} & = k J^{k - 1} \pd{g}{u^{n - 1}} - (k - 1) J^{k} \pd{g}{u^n}, \quad k = 2, \dots, n - 1, \\
0 & = J^{n - 1} \pd{g}{u^{n - 1}}.
\end{aligned}
\end{equation}   
\end{Lemma}
\begin{proof}
We prove the formula \eqref{mod2} by induction. For $k = 2$ the statement follows from \eqref{eq3}. Now assume that it holds for $k$, and let us show that it holds for $k + 1$. We have
$$
\begin{aligned}
\pd{g}{u^{n - k - 1}} & = 2 J \pd{g}{u^{n - k}} - J^2 \pd{g}{u^{n - k + 1}} = 2 J \Big( k J^{k - 1} \pd{g}{u^{n - 1}} - (k - 1) J^{k} \pd{g}{u^n} \Big) - \\ 
& - J^2 \Big( (k - 1) J^{k - 2} \pd{g}{u^{n - 1}} - (k - 2) J^{k - 1} \pd{g}{u^n}\Big) = (2k - k + 1) J^{k} \pd{g}{u^{n - 1}} - \\
& - (2(k-1)- (k-2)) J^{k + 1} \pd{g}{u^n} = (k + 1) J^{k} \pd{g}{u^{n - 1}} - k J^{k + 1} \pd{g}{u^n}.
\end{aligned}
$$
Thus, it holds for $k + 1$.   

We have shown that the $n - 2$ equations of \eqref{eq3} are equivalent to \eqref{mod2}. The second equation of \eqref{eq3} yields
$$
\begin{aligned}
J^{2} \pd{g}{u^2} - 2 J \pd{g}{u^1} & = J^{2} \Big( (n - 2) J^{n - 3} \pd{g}{u^{n - 1}} - (n - 3) J^{n - 2} \pd{g}{u^n} \Big) - \\
& - 2 J \Big( (n - 1) J^{n - 2} \pd{g}{u^{n - 1}} - (n - 2) J^{n  - 1} \pd{g}{u^{n}}\Big) = - n J^{n - 1} \pd{g}{u^{n - 1}} = 0.  
\end{aligned}
$$
Here we used $J^n = 0$. Finally, the first equation of \eqref{eq3} is a simple corollary of \eqref{mod2}:
$$
J^2 \pd{g}{u^1} = J^2 \Big( (n - 1) J^{n - 2} \pd{g}{u^{n - 1}} - (n - 2) J^{n - 1} \pd{g}{u^n} \Big) = 0.
$$
Thus, the lemma is proved.
\end{proof}

Let us show that the formulas \eqref{solution} indeed produce Nijenhuis operators. By construction $LJ - JL = 0$, there exist functions $g_1, \dots, g_n$ such that
$$
L = g_1 J^{n - 1} + \dots + g_n \Id.
$$
We denote the arguments of functions $f_i$ as $p, q$. We have
$$
\begin{aligned}
\pd{L}{u^n} & = \pd{}{u^n} \Big( f_1(P, Q) J^{n - 1} + \dots + f_{n - 1}(P, Q) J + f_n(P) \Big) = \\
& = \pd{f_1}{p} J^{n - 1} + \dots + \pd{f_{n - 1}}{p} J + \pd{f_n}{p} = \pd{g_1}{u^n} J^{n - 1} + \dots + \pd{g_n}{u^n} \Id, \\
\pd{L}{u^{n - 1}} & = \pd{}{u^{n - 1}} \Big( f_1(P, Q) J^{n - 1} + \dots + f_{n - 1}(P, Q) J + f_n(P) \Big) = \\
& = \pd{f_1}{q} J^{n - 1} + \dots + \pd{f_{n - 1}}{q} J + J \Big(\pd{f_1}{p} J^{n - 1} + \dots + \pd{f_n}{p}\Big) = \\
& = \pd{g_1}{u^{n - 1}} J^{n - 1} + \dots + \pd{g_n}{u^{n-1}} \Id.
\end{aligned}
$$
At the same time
$$
\begin{aligned}
\pd{L}{u^{n - k}} & = \pd{}{u^{n - k}} \Big( f_1(P, Q) J^{n - 1} + \dots + f_{n - 1}(P, Q) J + f_n(P) \Big) = \\
& = k J^{k - 1} \Big(\pd{f_1}{q} J^{n - 1} + \dots + \pd{f_{n - 1}}{q} J\Big) + J^k \Big(\pd{f_1}{p} J^{n - 1} + \dots + \pd{f_n}{p} \Id \Big) = \\
&  = k J^{k - 1} \Big( \pd{L}{u^{n - 1}} - J \pd{L}{u^n}\Big) + J^k \pd{L}{u^n} = k J^{k - 1} \pd{L}{u^{n - 1}} - (k - 1) J^k \pd{L}{u^n} =   \\
& = \pd{g_1}{u^{n - k}} J^{n - 1} + \dots + \pd{g_n}{u^{n-k}} \Id.
\end{aligned}
$$
Applying the r.h.s. to the basis vector field $\partial_n$, we get
$$
\begin{aligned}
\pd{L}{u^{n - k}} \partial_n = \Bigg( \pd{g_1}{u^{n - k}} J^{n - 1} + \dots + \pd{g_n}{u^{n - k}} \Id\Bigg) \partial_n = \pd{g_1}{u^{n - k}}\partial_1 + \dots + \pd{g_n}{u^{n - k}} \partial_n = \pd{g}{u^{n - k}}.
\end{aligned}
$$
Similar calculations for the l.h.s. yield
$$
\Big(k J^{k - 1} \pd{L}{u^{n - 1}} - (k - 1) J^k \pd{L}{u^n}\Big) \partial_n = k J^{k - 1} \pd{g}{u^{n - 1}} - (k - 1) J^k \pd{g}{u^n}.
$$
Thus, functions $g_i$ satisfy the first part of the system \eqref{mod2}. The last equation, namely $$ 
J^{n - 1} \pd{g}{u^{n - 1}} = 0,
$$
is equivalent to the fact that $\pd{g_n}{u^{n - 1}} = 0$; that is, $g_n$ does not depend on $u^{n - 1}$. 

The formula \eqref{solution} implies that $g_n$ is obtained from the decomposition $f_n(P)$. By direct computation we find that the function on the diagonal depends on $u^n$ only (see formula 6 before Theorem 1.4 in \cite{nij4}). In particular, it does not depend on $u^{n - 1}$ and, therefore, the functions $g_i$ satisfy the entire system \eqref{mod2}. Due to Theorem \ref{t1} and Lemma \ref{lm4} this means that $L$ is a Nijenhuis operator in Toeplitz form.

Now, let us show that we have constructed all the solutions. We start with a lemma.

\begin{Lemma}\label{lm5}
Fix the coordinates $u^1, \dots, u^n$ and consider $L$ to be the Nijenhuis operator in the Toeplitz form. Assume that all functions $g_i$ satisfy the condition
$$
g_i (0, \dots, 0, u^{n - 1}, u^n) \equiv 0, \quad i = 1, \dots, n,
$$
that is, the restriction of each $g_i$ onto the two-dimensional plane $u^1 = \dots = u^{n - 3} = 0$ is zero. Then $L = 0$.
\end{Lemma}
\begin{proof}
The first equation in \eqref{eq1} implies that $g_n$ does not depend on $u^1, \dots, u^{n - 2}$. At the same time, the condition of the lemma implies that $g_n \equiv 0$. The second equation in \eqref{eq1} in this case takes the form $J^{2*} \ddd g_{n - 1}$. Recalling previous arguments, we get $g_{n - 1} \equiv 0$. 

Proceeding in the same way, at each step we obtain an equation of the form $J^{2*} \ddd g_i = 0, i = n - 2, \dots, 1$. Thus, we get $g_i = 0, i = 1, \dots, n$, and the lemma is proved.
\end{proof}

Consider the decomposition
\begin{equation*}
f_i \big(u^n \Id + u^{n - 1} J, u^{n - 1} \Id\big) = f_{i, n} + f_{i, n - 1} J + \dots + f_{i, 1} J^{n - 1}
\end{equation*}
Here $f_{i, n} = f_i$ and each $f_{i, j}$ for $j < n$ is written in terms of $f_i$, its derivatives in $p$ and $u^1, \dots, u^n$. We obtain the following:
$$
\begin{aligned}
& g_1(0, \dots, 0, u^{n - 1}, u^n) J^{n - 1} + \dots + g_n(0, \dots, 0, u^{n - 1}, u^n) \Id = h_1(u^{n - 1}, u^n) J^{n - 1} + \dots + h_n(u^{n - 1}, u^n)Id = \\
= & f_1 \big(u^n \Id + u^{n - 1} J, u^{n - 1} \Id\big) J^{n - 1} + \dots + f_n \big(u^n \Id + u^{n - 1} J, u^{n - 1} \Id\big)Id = J^{n - 1} \Big(\sum_{i = 1}^n f_{1, i} J^{n - i}\Big) + \dots \\
+ & \sum_{i = 1}^n f_{n, i} J^{n - i} = f_{n, n} \Id + \Big( f_{n, n - 1} + f_{n - 1, n}\Big) J + \dots + \Big( f_{1, n} + \dots + f_{n, 1}\Big) J^{n - 1}.
\end{aligned}
$$
We obtain a system in the form
\begin{equation}\label{stt}
h_n = f_n, \quad h_{n - 1} = f_{n - 1} + f_{n, n - 1}, \; \ldots, \quad h_1 = f_1 + f_{2, n - 1} +\dots + f_{n, 1}.    
\end{equation}
This is a triangular system that establishes a one-to-one correspondence between $h_1, \dots, h_n$ and $f_1, \dots, f_n$.

Now consider an arbitrary $gl$-regular Nijenhuis operator in the upper triangular Toeplitz form $L$. Denote $h_i(u^{n - 1}, u^n) = g_i(0, \dots, 0, u^{n - 1}, u^n)$. Construct functions $f_1, \dots, f_n$, using the formula \eqref{stt}. Taking the solution for the initial conditions $f_i$ in the formula \eqref{solution}, we obtain a Nijenhuis operator in upper triangular Toeplitz form, which we denote $\bar M$. The restrictions of both operators on the plane $u^1 = \dots = u^{n - 2} = 0$ coincide; hence, due to the linearity of the system and Lemma \ref{lm5}, the operator field $M - \bar M$ is identically zero. Thus, formula \eqref{solution} yields all the solutions.


\section{Proof of Theorem \ref{t3} and Remark \ref{r1}}\label{proof3}

By construction, $M$ is in the form
$$
M = \left( \begin{array}{cccc}
     0 & m_{n - 1} & \dots & m_1  \\
     0 & 0 & \ddots & \vdots \\
     & \ddots &\ddots  & m_{n - 1}\\
     0 & 0 & \dots  & 0
\end{array}\right),
$$
where $m_{n - 1} \neq 0$. The Jacobi matrix $\pd{v}{u}$ of functions $v_1, \dots, v_n$ is upper triangular with 
$$
m_{n - 1}^s \pd{v^1}{u^1}, \quad s = 0, 1, \dots, n - 1
$$
on the diagonal. The condition $\pd{v_1}{u^1} \neq 0$ is equivalent to the non-degeneracy of the Jacobi matrix. Thus, $v^i$ defines a coordinate transformation. The system \eqref{sys} can be rewritten as
$$
\pd{v}{u} M = J \pd{v}{u}.
$$
and, thus, $\pd{v}{u} M \pd{v}{u}^{-1} = J$. In particular, $M$ takes the form $J$ in coordinates $v^1, \dots, v^n$.

As $m_{n - 1} \neq 0$, operator $M$ is $gl$-regular. By construction $MJ - JM = 0$, there exist $s_1, \dots, s_n$ such that $J = s_1 M^{n - 1} + \dots + s_n \Id$. Now consider an arbitrary $L$ in the upper triangular Toeplitz form, that is, $L = g_1 J^{n - 1} + \dots + g_n \Id$. Substituting the expression for $J$ via $M$, we get that $L = \bar g_1 M^{n - 1} + \dots + \bar g_n \Id$. As a result, we obtain the equality
$$
\pd{v}{u} L \pd{v}{u}^{-1} = \bar g_1 J^{n - 1} + \dots + \bar g_n \Id,
$$
that is, $L$ in coordinates $v^1, \dots, v^n$ is in the upper triangular Toeplitz form. Thus, the second statement of the theorem \ref{t3} is proved.

Now consider an arbitrary transformation $v(u)$ that preserves the upper triangular Toeplitz form of $L$. In coordinates $v^1, \dots, v^n$, we take the operator field $J$. Since $L$ is $gl$-regular, by a similar argument as above, we write $J = q_1 L^{n - 1} + \dots + q_n \Id$ for some functions $q_i$. 

Applying the inverse transformation $u(v)$, we find that $L$ is in upper triangular Toeplitz form. In particular, this means that $\pd{u}{v} J \pd{u}{v}^{-1} = M$ is a Nijenhuis operator in the upper triangular Toeplitz form with zeros on the diagonal. Applying theorem \ref{t2}, we find that $M$ can be constructed from $P, Q$ by making an appropriate choice of functions $f_1, \dots, f_{n - 1}$.

In coordinates $v$ functions $\ddd v^i$ satisfy \eqref{sys} for $M = J$. At the same time, this system is defined coordinate-free, so it is satisfied in the initial coordinates $u^1, \dots, u^n$ as well. Thus, we get that an arbitrary coordinate transformation, which preserves the upper triangular Toeplitz form of $L$, is a solution of the system \eqref{sys}. The non-degeneracy of the Jacobi matrix yields the condition $\pd{v^1}{u^1} \neq 0$. Thus, the third statement of theorem \ref{t3} is proved.

Now let us proceed with a proof of remark \eqref{r1}. First, notice that if we have two coordinate changes $v(u)$ and $w(u)$ that transform the given operator $M$ into $J$, then $w(v)$ obviously transforms $J$ into $J$; that is, it preserves the normal form of the operator field. Thus, we need to describe such transformations.

If the symmetric (in lower indices) part of the tensor $\langle L, M \rangle$ defined by the formula \eqref{ii:eq2} identically vanishes, then $L, M$ are called symmetries of one another. If the entire tensor vanishes, then they are called strong symmetries of one another. We say that $gl$-regular Nijenhuis operator $L$ is in companion form if
$$
L = \left(\begin{array}{ccccc}
     \sigma_1 & 1 & 0 & \dots & 0  \\
     \sigma_2 & 0 & 1 & \dots & 0  \\
     & & & \ddots & 0 \\
     \sigma_{n - 1} & 0 & 0 & \dots & 1 \\
     \sigma_n & 0 & 0 & \dots & 0
\end{array}\right).
$$
Here $\sigma_i$ are the coefficients of the characteristic polynomial. For nilpotent $L$ this form coincides with $J$.

The procedure described in remark \ref{r1} produces all strong symmetries of $J$ (Theorems 1.2 and 1.4 in \cite{nij4}). Moreover, the condition $h_1'(0) \neq 0$ is equivalent to the fact that the functions $w_i$ can be taken as coordinates (the argument is the same as in the beginning of the proof of the theorem \ref{t3}). These coefficients, in turn, are in one-to-one correspondence with the coordinate transformations $w(u)$, which transform $J$ into the first companion form. As mentioned earlier, this form of the Jordan block coincides with the Jordan block itself.

Thus, we get that these are all coordinate transformations that preserve $J$. The statement of the remark has been proven.


\section{Proof of Theorem \ref{t4}}\label{proof4}

We start with a Lemma.

\begin{Lemma}\label{krp1}
Let $M$ be a Nijenhuis operator from the statement of the Theorem. Assume that for 1-form $\omega$, both 1-forms $M^*\omega$ and $(M^*)^2 \omega$ are closed. Then $\ddd \omega (M\xi, M\eta) = 0$ for all vector fields $\xi, \eta$.
\end{Lemma}
\begin{proof}
Recall that the Nijenhuis torsion $\mathcal N_M : \Omega^1(\mathrm M^n) \to \Omega^2(\mathrm M^n)$ can be treated as a mapping from 1-forms $\Omega^1(\mathrm M^n)$ to 2-forms $\Omega^2(\mathrm M^n)$ \cite{nij1}:
$$
\beta( \cdot\, , \cdot ) = \ddd(M^*\omega) (M\cdot\, , \cdot) + \ddd(M^*\omega) (\cdot\, , M\cdot) - \ddd({M^*}^2\omega) (\cdot\, ,\cdot) - \ddd\omega (M\cdot\, ,M\cdot).  
$$
Here $\omega \in \Omega^1(\mathrm M^n), \beta = \mathcal N_M \omega \in \Omega^2(\mathrm M^n)$. The statement of the Lemma immediately follows from this formula.
\end{proof}

Let us proceed to the algorithm. For a given initial data of $M$ and step 0, it produces $n$ functions $v^1, \dots, v^n$. Thus, it is obviously correct.

Now let us proceed to the second statement of the Theorem. First, notice that during the initial run $1-$form $\omega$ with condition $M^* \omega = \ddd v^n$ has only one non-zero term. Thus, the formula \eqref{form2} contains only one summand and, by construction, $\pd{v^{n - 1}}{u^{n - 1}} = \omega_{n - 1}$ and $M^* \ddd v^{n - 1} = \ddd v^n$. At the same time $\omega_{n - 1} g_{n - 1} = q(u^n)$ and $\omega_{n - 1} \neq 0$.

Now assume that we have constructed functions $v^{n - i}, \dots, v^n$ that satisfy $M^* \ddd v^s = \ddd v^{s + 1}, s = n - i, \dots, n - 1$. There are at least two such functions. This implies that $\omega$, constructed in step 2, satisfies the conditions of Lemma \ref{krp1}. The statement of the Lemma implies that
$$
\pd{\omega_i}{u^j} = \pd{\omega_j}{u^i}, \quad 1 \leq i, j \leq n - 1.
$$
This means that function $v^{n - i - 1}$, constructed in step 3, satisfies the conditions
$$
\pd{v^{n - i - 1}}{u^i} = \omega_i, \quad i = 1, \dots, n - 1.
$$
In particular, we get $M^* \ddd v^{n - i - 1} = \ddd v^{n - i}$ and $\omega_{n - i - 1} \neq 0$. Thus, the algorithm indeed produces the solutions of \eqref{sys}, and in the last step, we obtain $\pd{v^1}{u^1} = \omega_1 \neq 0$.

Finally, we proceed to the third statement of the theorem. We will need the following lemma.

\begin{Lemma}\label{krp2}
Consider $v^1, \dots, v^n$ to be the solution of \eqref{sys}. Assume that for all $i = 1, \dots, n$ the functions $v^i(0, \dots, 0, u^n)$ are constants. Then all the $v^i$ are constants.    
\end{Lemma}
\begin{proof}
First, notice that $M^* \ddd v^n = 0$ and, thus, $\ddd v^n = \pd{v^n}{u^n} \ddd u^n$. The closedness of the 1-form implies that $v^n$ depends only on $u^n$ and, in particular, the function $v^n(0, \dots, 0, u^n) = v^n$ is constant. This, in turn, implies that $M^* \ddd v^{n - 1} = \ddd v^n = 0$. Repeating the same argument, we get that $v^{n - 1}, \dots, v^1$ are all constants. The lemma is proved.  
\end{proof}

Now consider $\bar v^1, \dots, \bar v^n$ to be an arbitrary solution of \eqref{sys} with $\pd{v^1}{u^1} \neq 0$. Denote the sequence $r_i(u^n) = \bar v^i(0, \dots, 0, v^n)$. Using the algorithm, let us construct the solution of \eqref{sys}, taking $r_i$ to be the "constant of integration" $r(u^n)$ in formula \eqref{form2} for $i = n - 1, \dots, 1$ and $\int q(\tau) \ddd \tau = r_n(u^n)$. In particular, $q(u^n) \neq 0$ ensured that $\pd{v^1}{u^1} \neq 0$.

Due to the formula \eqref{form2} we have $v^i(0, \dots, 0, u^n) = r_i(u^n)$. Since the system \eqref{sys} is linear, the difference $\bar v^i - v^i$ yields a solution as well. At the same time, it satisfies the conditions of lemma \ref{krp2} for all constants being zero. Thus, we get that $\bar v^i \equiv v^i$ for all the points. The theorem is completely proven.


\end{document}